\documentclass[11pt]{amsart}
\usepackage[cm]{fullpage}
\usepackage{amssymb,amsmath,amsthm,amscd,mathrsfs,graphicx}
\usepackage[cmtip,all]{xy}
\usepackage{pb-diagram}
\usepackage{tikz}

\numberwithin{equation}{section}
\newtheorem{teo}{Theorem}[section]

\newtheorem{lem}[teo]{Lemma}

\newtheorem{theorem}{Theorem}[section]
\newtheorem{proposition}{Proposition}[section]
\newtheorem{lemma}{Lemma}[section]

\newcommand{\ov}[1]{\overline{#1}}

\newcommand{\ve}{\varepsilon}

\theoremstyle{definition}

\theoremstyle{remark}
\newtheorem{remark}[teo]{Remark}

\begin{document}
\bibliographystyle{amsplain}

\title{Counterexamples to quasiconcavity \\ for the heat equation}

\author[A. Chau]{Albert Chau}
\address{Department of Mathematics, The University of British Columbia, 1984 Mathematics Road, Vancouver, B.C.,  Canada V6T 1Z2.  Email: chau@math.ubc.ca. } 
\author[B. Weinkove]{Ben Weinkove}

\address{Department of Mathematics, Northwestern University, 2033 Sheridan Road, Evanston, IL 60208, USA.  Email: weinkove@math.northwestern.edu.}

\thanks{Research supported in part by  NSERC grant $\#$327637-06 and NSF grants DMS-1406164 and DMS-1709544}

\maketitle

\begin{abstract}
We construct solutions to the heat equation on convex rings showing that quasiconcavity may not be preserved along the flow,
even for smooth and subharmonic initial data.
\end{abstract}

\section{Introduction}

Let $\Omega_0$ and $\Omega_1$ be convex  open sets with smooth boundary in $\mathbb{R}^n$ with $\ov{\Omega}_1 \subset \Omega_0$.  Assume that $\Omega_1$ contains the origin.  Denote by  $\Omega = \Omega_0 \setminus \ov{\Omega}_1$ the open convex ring.   We say that a function $u(x)$ on $\ov{\Omega}$ is \emph{quasiconcave} if the sets 
$$\{ x \in \ov{\Omega} \ | \  u(x) \ge c \} \cup \Omega_1$$
 are convex subsets of $\mathbb{R}^n$ for every $c\in \mathbb{R}$.  Fix $T$ with $0<T\le \infty$.  A function $u=u(x,t)$ on $\ov{\Omega} \times [0,T)$ is called \emph{space-time quasiconcave} if the sets $$\{ (x,t) \in \ov{\Omega} \times [0, T) \ | \ u(x,t) \ge c \}\cup (\Omega_1 \times [0,T))$$ are convex subsets of $\mathbb{R}^{n+1}$ for every $c \in \mathbb{R}$.  Note that space-time quasiconcavity of $u(x,t)$ implies that $x \mapsto u(x,t)$ is quasiconcave on $\ov{\Omega}$ for each $t$.

It is a classical result that if $u$ is a harmonic function on $\Omega$ satisfying the Dirichlet boundary conditions
\begin{equation} \label{dir}
u|_{\partial \Omega_0} =0, \quad u|_{\partial \Omega_1} =1
\end{equation}
 then $u$ is quasiconcave \cite{A, G, L}.  This result has been extended to solutions $u$ of more general elliptic PDEs, where there is a general principle that  convexity properties of $\Omega_0$ and $\Omega_1$ imply convexity of the superlevel sets of $u$.  These results are proved via ``macroscopic'' approaches involving functions of two points which could be far apart, or ``microscopic'' approaches using functions of the principal curvatures of the level sets together with constant rank theorems.  See for example \cite{ALL, BG,  BGMX, BLS, BL, CF, CGM, CS, CMY, Ka, Ko2, KL, RR, S, SWYY, SW} and the references therein.  On the other hand, convexity properties fail for solutions to some elliptic PDEs \cite{HNS, Wa}.

There has been considerable interest in \emph{parabolic} versions of these classical results.   Parabolic constant rank theorems in rather general contexts have been established by Hu-Ma \cite{HM}, Chen-Hu \cite{CH} and Chen-Shi \cite{ChS}.   An older result of Borell \cite{Bo} assumes that the initial data is 
 identically zero and shows that the solution $u(x,t)$ to the heat equation with boundary conditions (\ref{dir}) is space-time quasiconcave.  This result has been extended to more general parabolic equations by Ishige-Salani \cite{Ish2, Ish3}.  
  However, the  assumption of identically vanishing initial data is rather restrictive.   This begs the question: what assumption on the initial data is necessary to ensure space-time quasiconcavity of the solution to the heat equation?  Ishige-Salani \cite{Ish} gave examples to show that quasiconcavity of the initial data is not sufficient.  A natural condition considered in \cite{DK0, DK} is that $u_0$ in addition be subharmonic (with sufficient regularity), namely $\Delta u_0 \ge 0$.  In this paper we provide a counterexample to show that this is still not sufficient to ensure quasiconcavity of the solution to the heat equation.

More precisely, we consider a classical solution $u$ of the following problem:
 \begin{equation} \label{heat}
 \left\{ \begin{array}{ll} \partial u/\partial t = {}  \Delta u, \quad & \textrm{on } \Omega \times (0, T) \\
 u(x,0) = u_0(x), \quad & x\in \Omega \\
 u(x,t) = 0, \quad &  (x,t) \in \partial \Omega_0 \times [0,T) \\
 u(x,t) = 1, \quad &  (x,t) \in \partial\Omega_1 \times [0,T), \end{array} \right.
 \end{equation}
 for $0<T\le \infty$.  Here $u_0$ is  a smooth function  on $\overline{\Omega}$ which satisfies the conditions
  \begin{equation} \label{ass}
  \begin{split}
  & 
 u_0=1 \  \textrm{on } \partial{\Omega}_1, \ \ u_0 =0 \ \textrm{on } \partial \Omega_0, \ \ x \cdot \nabla u_0(x) \le 0 \ \textrm{on } \Omega \\ 
 & \Delta u_0 \ge 0 \ \textrm{on } \Omega \quad \textrm{but not identically zero.}
 \end{split}
 \end{equation} 
 We call such a function $u_0$ \emph{admissible}.

We consider the following question:  if an admissible $u_0$ is quasiconcave on $\ov{\Omega}$, does it follow that the solution $u(x,t)$ to (\ref{heat}) is space-time quasiconcave on $\ov{\Omega} \times [0,T)$?  If not, is $x\mapsto u(x,t)$ quasiconcave on $\ov{\Omega}$ for each $t>0$?

We construct a counterexample to show that the answer to both of these questions is negative.

\begin{theorem} \label{theorem1} For any $n\ge 2$, let $\Omega_1$ and $\Omega_0$ be balls in $\mathbb{R}^n$ centered at the origin, of radii 1 and 2 respectively, so that $\Omega$ is the annulus $1< r <2$.
There is an admissible quasiconcave function $u_0$ with the following properties:
 \begin{enumerate}
\item[(i)] The solution $u(x,t)$ to (\ref{heat}) is smooth on $\Omega \times (0,\infty)$ and continuous on $\overline{\Omega} \times [0,\infty)$.
\item[(ii)]  There exists  $t_0>0$ such that the function $x \mapsto u(x,t_0)$ fails to be quasiconcave on $\overline{\Omega}$.  
\end{enumerate}
\end{theorem}

This example implies that the statement of \cite[Theorem 3]{DK} (see the discussion in \cite[Section 7]{Ish2}) requires additional hypotheses.

Our construction in Theorem \ref{theorem1} is based on the simple observation that the union of interiors of a sphere and a non-spherical ellipsoid is non-convex unless one is contained in the other.  We use this observation as follows.   We first find a radially symmetric admissible function $V$ which is close to 1 near the boundary of $\Omega_1$ and drops off rapidly to zero.  For every positive time, the level sets of the heat flow solution starting from $V$ will then give a foliation of $\Omega$ by spheres.  We then construct an admissible function $W$ whose level sets are spherical near the boundary of $\Omega_1$ but non-spherical ellipsoids as one goes outwards.  We choose $u_0 = (1-\ve)V + \ve W$, for $\ve>0$ small,  as initial data.  The relatively large radially symmetric heat distribution of $(1-\ve)V$ quickly emanates out and interacts with the ellipsoidal level sets of $\ve W$ to give a non-convex superlevel set after some positive time.  The proof of Theorem \ref{theorem1}, given in Section \ref{section2}, makes this heuristic argument precise.

Note that by necessity our counterexample is not radially symmetric, and must have dimension $n>1$.  If radial symmetry is imposed for $u_0$, which implies quasiconcavity of $x\mapsto u(x,t)$ for each $t$ (since the superlevel sets are balls in $\mathbb{R}^n$) it is natural to ask whether the stronger condition of space-time quasiconcavity follows.  Our next counterexample shows that the answer to this is again negative for any $n\geq 1$.

\pagebreak[3]

\begin{theorem} \label{theorem2}  For any $n\geq 1$, let $\Omega_1$ and $\Omega_0$ be  balls in $\mathbb{R}^n$ of radii $R$ and $R+1$ respectively, for a constant $R>1$.  For $R$ sufficiently large, there is an admissible function $u_0$ on $\overline{\Omega}= \{ R \le r \le R+1\} \subset \mathbb{R}^n$  with the following properties:
 \begin{enumerate}
 \item[(i)] $u_0$ is radially symmetric  and hence if $u(x,t)$ solves (\ref{heat}) then $x \mapsto u(x,t)$ is quasiconcave on $\overline{\Omega}$ for every $t\ge 0$. 
\item[(ii)] $u(x,t)$ is smooth on $\ov{\Omega} \times [0,\infty)$.
\item[(iii)]  $u(x,t)$ is not space-time quasiconcave on $\overline{\Omega} \times [0,T)$ for any $T>0$.
\end{enumerate}
\end{theorem}

 In Theorem \ref{theorem2} any space-time level set $\partial \Omega_c := \{ (x,t) \in \ov{\Omega} \times [0, \infty) \ | \ u(x,t) =c \}$ for $c\in (0, 1)$ will be given by a graph $t=f(|x|)$ where $f(r)$ is a smooth strictly increasing function defined on some interval $[r_0, r_1)$ where  $f(r_0)=0$.  In particular, $f$ is defined implicitly by $u(r, f(r))=c$ and differentiating this and using \eqref{heat} gives  
$$f''(r)=\frac{-1}{u_t}(u_{rr}+2 u_{rt}f' +  u_{tt}(f')^2 )$$
We show that by solving an ordinary differential equation, we may choose the function $u_0$ so that the right hand side above is negative at $(r_0, 0)$, implying $f''(r_0)<0$ and thus $\partial \Omega_c $ is not convex.  The details of this argument are given in Section \ref{section3} where we prove Theorem \ref{theorem2}. 

Finally, in Section \ref{section4} we give a different counterexample to space-time quasiconcavity using a ``two-point function'' as in \cite{W} and inspired by the work of Rosay-Rudin \cite{RR}.  It satisfies properties (i), (ii) of Theorem \ref{theorem2}, but (iii) must be replaced by
\begin{enumerate}
\item[(iii)*] $u(x,t)$ is not space-time quasiconcave on $\overline{\Omega}_0 \times [0,T)$ for {\bf some} $T>0$.
\end{enumerate}
The argument using the two-point function is perhaps slightly more intuitive than that of Theorem \ref{theorem2} and the counterexample is defined on the annulus $\{1 < r < 2\}$.   

\bigskip
\noindent
{\bf Acknowledgements.} \ The authors thank the referee for correcting some inaccuracies in a previous version of this paper.

\pagebreak[3]

\section{A counterexample to quasiconcavity} \label{section2}

Let $\Omega_0$, $\Omega_1$ and $\Omega$ be as in the introduction.  We first gather some well-known facts about solutions to (\ref{heat}).

\begin{proposition}  \label{prop} Let $u_0$ be an admissible function on $\ov{\Omega}$.  Then there exists a unique continuous solution $u(x,t)$ to (\ref{heat}) on $\overline{\Omega} \times [0,\infty)$ which is smooth on $\Omega \times (0,\infty)$ and satisfies the following conditions for $(x,t) \in \Omega \times (0,\infty)$ 
\begin{enumerate}
\item[(i)] $\displaystyle{0 < u(x,t) < 1}$.
\item[(ii)] $\displaystyle{u_t(x,t)=\Delta u(x,t) >0}$.
\item[(iii)] $\displaystyle{x \cdot \nabla u(x,t) <0}$.
\end{enumerate}
Moreover, as $t \rightarrow \infty$, $u(x,t)$ converges smoothly on $\Omega$ to the harmonic function $u_{\infty}$ with boundary conditions $u_{\infty}|_{\partial \Omega_0}=0$ and $u_{\infty}|_{\partial \Omega_1}=1$.
\end{proposition}
\begin{proof}
The existence of a unique solution $u(x,t)$ to (\ref{heat}) with the stated regularity, and the convergence as $t \rightarrow \infty$ are classical, see for example \cite{F}.
From (\ref{ass}) we have $0 \le u_0 \le 1$ and then (i) follows from the strong maximum principle for parabolic equations.  Parts (ii) and (iii) are proved in \cite[Lemma 1]{DK} and are also consequences of the maximum principle.   \end{proof}

We now start the proof of Theorem \ref{theorem1}.  Let $\Omega = \{ 1 < r<2 \}$ be as in the statement of the theorem and we assume for the rest of this section that $n \ge 2$.
  Part (i) of Theorem \ref{theorem1} is a consequence of the above proposition.  For part (ii), we begin by defining two auxiliary functions $V_{\rho}$ and $W$.

\begin{lem}\label{L1} For any $\rho \in (1,3/2]$, define a radially symmetric function $V_{\rho}=V_{\rho}(r)$ on $\ov{\Omega}$ by
$$V_{\rho}(r) = \left\{ \begin{array}{ll} \exp \left( \frac{n}{r-\rho} - \frac{n}{1-\rho} \right),  & 1 \le r <\rho \\ 0, & \rho \le r \le 2. \end{array} \right.$$
Then $V_{\rho}$ is an admissible function. 
\end{lem}
\begin{proof}

We drop the $\rho$ subscript.  Observe that $V(r) \ge 0$ is smooth, decreasing on $[1,2]$,  satisfies $V(1)=1$ and
\begin{equation} \label{DeltaU}
\begin{split}
(\Delta V)(r)= {} & V''(r) + \frac{(n-1) V'(r)}{r} \\ = {} & n V(r) \left( \frac{2r(r-\rho) + nr - (n-1)(r-\rho)^2}{r(r-\rho)^4} \right) \\  \ge {} &  \frac{n}{4r(r-\rho)^4} V(r) \ge 0
\end{split}
\end{equation}
where the second-to-last inequality follows from the inequalities $|r-\rho| \le 1/2$ and $1 \le r \le 3/2$ when $r <\rho$.
\end{proof}

We now use $V_{\rho}$ to define a non-radially symmetric function $W$.

\pagebreak[3]
\begin{lem}\label{L2}  Fix $r_0, r_1$  with $1< r_0 <r_1\leq 3/2$.  There exists an admissible function $W$ on $\ov{\Omega}$ with the following properties.
\begin{enumerate}
\item[(i)] $W$ is radially symmetric on $1\le r \le r_0$.
\item[(ii)] There exists a smooth strictly decreasing function $b$ on $[r_0,r_1]$ with $b(r_0)=1$ and $b(r_1)  \in (0,1)$ such that if we define $E_R$ to be the ellipsoid with equation
\begin{equation} \label{ER}
b(R)^2  x_1^2 +x_2^2 + \cdots + x_n^2 = R^2, \quad \textrm{for } R \in [r_0, r_1],
\end{equation}
then the level sets $\{ W = c \}$ for $0 < c < W(r_0)$ are the non-spherical ellipsoids $E_R$ for $R \in (r_0, r_1)$.
\item[(iii)] $W$ vanishes outside the ellipsoid $E_{r_1}$.
\end{enumerate}
\end{lem}
\begin{proof}  Let $V=V_{r_1}$ be as in Lemma \ref{L1} defined with  $\rho=r_1$.  Regarding $V$ as a function of $x_1, \ldots, x_n$ we compute for any $i, j=1,...,n$:
$$V_{x_i} =  -V(r) \frac{nx_i}{r(r-r_1)^2}$$
and
$$V_{x_ix_j} = V(r) \left( \frac{n^2x_i x_j}{r^2(r-r_1)^4} + \frac{2nx_i x_j}{r^2(r-r_1)^3} - \frac{n(r^2\delta_{ij} -x_i x_j)}{r^3(r-r_1)^2}  \right).$$ 
It follows from this and (\ref{DeltaU}) that for some constant $\beta=\beta(n)$,
\begin{equation}\label{eeeee1} \Delta V + \sum_{i,j} c_{i,j} V_{x_ix_j}+ \sum_i c_i V_{x_i}  \geq 0\end{equation}
as long as $|c_{i, j}|, |c_i| \le \beta$.

Fix a smooth non-decreasing function $a(r):[1,2] \to [0,1]$ with $a(1)=0$ and $a(2)=1$ and with 
$\{ r \ | \ a'(r) >0 \}= (r_0, r_1)$.   Let $b(r)=1- \kappa a(r)$ for some constant $\kappa\in (0,1)$ to be determined.  Treating $b$ as a rotationally symmetric function of $x_1, \ldots, x_n$ consider the map $(y_1, \ldots, y_n) = (x_1 /b, x_2, \ldots,  x_n)$, which we will write as $y=\Psi(x)$.  Note that $\Psi$ is invertible and $\Psi(x) \to \textrm{Id}$ as $\kappa \to 0$ where the convergence is uniform in any $C^k$ norm on $\overline{\Omega}$.  It follows from the inverse function theorem that we likewise have $\Psi^{-1}(y)|_{\ov{\Omega}} \to \textrm{Id}$ uniformly in any $C^k$ norm on $\overline{\Omega}$ as $\kappa \to 0$.

The map $\Psi(x)$  is the identity on  $\{ 1 \le r \le r_0 \}$ and 
takes the sphere $x_1^2+ \cdots + x_n^2=R^2$ to the ellipsoid $b(R)^2 y_1^2 + y_{2}^2 + \cdots + y_n^2=R^2$ which is non-spherical exactly when $R\in (r_0, 2]$.  
 We choose $\kappa$ sufficiently small so that the ellipsoid $(1-\kappa)^2 y_1^2+ y_2^2 +  \cdots +  y_n^2=(r_1)^2$ is contained inside the sphere of radius 2.

Define $W: \ov{\Omega} \rightarrow \mathbb{R}$ by  $W(y)=V(x(y))$ for $x(y)= \Psi^{-1}(y)$.  Note that $W=0$ outside the ellipsoid $(1-\kappa)^2 y_1^2+ y_2^2 + \cdots +  y_n^2=(r_1)^2$.   Thus $W(y)$ satisfies  (i), (ii) and (iii) of the Lemma.

To see that $y \cdot (\nabla W)(y) \le 0$ we compute, 
$$y \cdot (\nabla W)(y) = \sum_{i,j} y_i V_{x_j}(x(y)) \frac{\partial x_j}{\partial y_i}= - (1+E)\frac{nr}{(r-r_1)^2}(x(y)) V(x(y)) $$
where $E$ is an ``error'' term which converges uniformly to zero as $\kappa$ tends to zero.  Hence $y \cdot \nabla W \le 0$ for $\kappa$ sufficiently small.

 All that remains is to show now is that $\Delta W \geq 0$ in $\Omega$.  We compute
 \[
 \begin{split}
 (\Delta W)(y) = {} & \sum_{i,j,k} V_{x_i x_j}(x(y)) \frac{\partial x_i}{\partial y_k} \frac{\partial x_j}{\partial y_k}+ \sum_{i,k} V_{x_i}(x(y)) \frac{\partial^2 x_i}{\partial y_k^2} \\
 = {} & (\Delta V)(x(y)) + \sum_{i,j} c_{i,j} V_{x_ix_j}(x(y))+ \sum_{i} c_i V_{x_i}(x(y)), 
 \end{split}
 \]
 for $c_{i,j}$ and $c_i$ which converge uniformly to zero as $\kappa$ tends to zero.  From \eqref{eeeee1} it follows that $\Delta W \ge 0$ for $\kappa$ sufficiently small.
\end{proof}

Next we have an elementary lemma about radially symmetric subharmonic functions.

\begin{lem} \label{lemmael}
Let $f=f(r)$ be a smooth radially symmetric function on $\Omega$ with $f(1)=1$, $f(2)=0$, $f_r \le 0$ and $\Delta f \ge 0$.  Fix $r_0, r_1$ with $1<r_0<r_1<2$.  Then 
$$f(r_1) \le (1-\sigma) f(r_0),$$
for $\sigma = (r_1-r_0)/(2  +r_1-2r_0)>0$.
\end{lem}
\begin{proof}
We begin by showing 
\begin{equation} \label{fr}
-f_r \ge \frac{1}{2-r_0} f, \quad \textrm{ on } [r_0, r_1].
\end{equation} Indeed the condition $\Delta f \ge 0$ implies that $r^{n-1}(- f_r(r))$ is nonincreasing in $r$.  Hence for $s \in [r_0, r_1]$ we have
$$(2-s) s^{n-1} (-f_r(s)) \ge \int_s^{2} r^{n-1}(- f_r(r)) dr \ge s^{n-1} \int_s^2 (-f_r(r))dr = s^{n-1} f(s),$$
where for the final equality we used $f(2)=0$, and (\ref{fr}) follows.

Next compute
$$f(r_0) - f(r_1) = \int_{r_0}^{r_1} (-f_r(r))dr \ge \frac{1}{2-r_0} \int_{r_0}^{r_1} f(r) dr \ge \frac{(r_1-r_0)}{2-r_0}  f(r_1),$$
where we recall for the last inequality that $f$ is decreasing in $r$. 
The result follows.
\end{proof}

We can now complete the proof of Theorem \ref{theorem1}.

\begin{proof}[Proof of Theorem \ref{theorem1}]
Let $V=V_{5/4}$ be as in Lemma \ref{L1}.  Let $W$ be as in Lemma \ref{L2} with $r_0=5/4$ and $r_1=3/2$.  We will write $v(t)$ and $w(t)$ for the solutions to (\ref{heat}) with initial conditions $V$ and $W$ respectively. 
For $\ve \in (0,1)$ to be determined, define
$$u_0 = (1-\ve) V + \ve W,$$
so that $u(t) = (1-\ve) v(t) + \ve w(t)$ is the solution of (\ref{heat}) starting at $u_0$.  Clearly $u_0$ is an admissible function.  In addition, note that for $\{ 1 < r \le 5/4 \}$ the level sets of $u_0$ are spheres and, since $V$ vanishes for $r>5/4$, the level sets for $u_0$ on $\{ r >5/4 \}$ are the same as those of $W$.  In particular, the level sets of $u_0$ are convex and so $u_0$ is quasiconcave.

Let $\eta_1, \eta_2 \in (0,1/4)$ be small constants to be determined and write $R^- = 3/2 - \eta_1$.  Then on the non-spherical ellipsoid $E_{R^-}$ with equation given by (\ref{ER}),  $W$ takes a constant value, $W=s>0$, say, which depends on $\eta_1$.  We have $W=0$ on and outside the ellipsoid $E_{3/2}$.  Define $R^+ = 3/2+\eta_2$ and write 
 $S_{R^+}$ for the sphere of radius $R^+$.  Observe that for $\eta_2$ sufficiently small  we can find $X \in S_{R^+}$ and $Y' \in E_{3/2}$ such that $(X+Y')/2$ lies outside both $S_{R^+}$ and $E_{3/2}$.  Next by choosing $\eta_1$ sufficiently small we can find a point $Y \in E_{R^-}$ close to $Y'$ so that $Z=(X+Y)/2$ lies outside both $S_{R^+}$ and $E_{3/2}$.  See Figure \ref{fig}.   Note that now $\eta_1$ is chosen, $s$ is a fixed positive number.

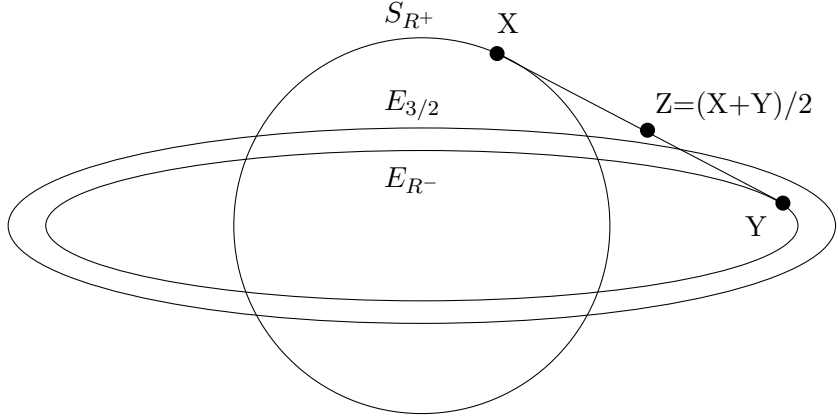
\begin{figure}[h] 

\begin{tikzpicture}

\draw (2,2) ellipse (5cm and 1cm);

\draw (2,2) ellipse (5.5cm and 1.3 cm);

 \node[text width=3cm] at (4.5,4.7) 
    {X};

\draw (2,2) circle (2.5cm);

\draw (3,4.29) -- (6.8,2.3);

\fill[black!40!black] (6.8,2.3) circle (.1cm);

\fill[black!40!black] (5.0,3.27) circle (.1cm);

\node[text width=3cm] at (6.6,3.6) 
    {Z=(X+Y)/2};

\node[text width=3cm] at (7.8,2.0) 
    {Y};

\node[text width=3cm] at (3,4.8) 
    {$S_{R^+}$};

\node[text width=3cm] at (3,3.6) 
    {$E_{3/2}$};

\node[text width=3cm] at (3,2.6) 
    {$E_{R^-}$};

\fill[black!40!black] (3,4.29) circle (.1cm);

\draw (6.8,2.3) circle (.05cm);

\end{tikzpicture}
\caption{Schematic diagram showing the ellipsoids $E_{R^-}$ and $E_{3/2}$, the sphere $S_{R^+}$ and the points $X,Y,Z$.}
\label{fig}
\end{figure}

Since $V(X)=W(Z)=0$, and using Proposition \ref{prop}, there exist continuous functions $\alpha(t), \beta(t)$ which vanish at $t=0$ and are positive for $t>0$ such that
  $$v(X,t) = s \alpha(t),  \quad w(Z,t) = \beta(t).$$
 Next we use Lemma \ref{lemmael} to see
  that there exists a constant $\sigma>0$ independent of $t$ such that
 $$v(Z,t) \le s (1-\sigma) \alpha(t).$$
 Indeed, the radially symmetric function $v(\cdot, t)$ satisfies the conditions of Lemma \ref{lemmael} and $Z$ lies at a fixed distance outside the sphere $S_{R^+}$ which contains $X$.

Next note that since $W(Y) =s$ and $w_t \ge 0$, we have $$w(Y,t) \ge s$$ for all $t\ge 0$.
Choose a small time $t_0>0$ such that 
$$\ve:= \alpha(t_0) <1/2, \quad \textrm{and} \ \beta(t_0) < \frac{\sigma s}{2}.$$
Compute that 
$$u(Z, t_0) \le (1-\ve) s (1-\sigma) \ve  +  \frac{\ve \sigma s}{2}   < (1-\ve) s \ve.$$
But 
$$u(X, t_0) \ge (1-\ve) s \ve, \quad u(Y, t_0) \ge \ve s > (1-\ve ) s \ve.$$
Hence $X$ and $Y$ lie in the superlevel set $\{ P \ | \ u(P, t_0) \ge (1-\ve) s \ve \}$ but $Z=(X+Y)/2$ does not, showing that this superlevel set is not convex.
\end{proof}

\section{A radially symmetric example} \label{section3}

In this section we give the proof of Theorem \ref{theorem2}.   First, define a smooth function $h:[0, 1]\to \mathbb{R}$ to have the following properties
\begin{enumerate}
\item[(a)] $1/10 \ge h(r) >0$ for $r \in (0,1)$.
\item[(b)] $h(r) = 1/10$ for
 $r \in [1/4, 3/4]$.
\item[(c)] $h^{(k)}(0)=0=h^{(k)}(1)$ for all $k=0,1,2,\ldots$.
\end{enumerate}
Next, for a constant $R>1$, define $v_R : [0,1] \rightarrow \mathbb{R}$ to be the solution of the Dirichlet problem
 \begin{equation} \label{dpv}
  v_R''(r)+(n-1) \frac{v'_R(r)}{r+R} =h(r), \  
  0<r<1, \quad v_R(0)=1, \ v_R(1)=0.
  \end{equation}  
We will determine the constant $R$ later.
By standard elliptic estimates \cite{GT} we have that $v_{R}(r)$ and its derivatives  are bounded on $[0, 1]$ uniformly with respect to $R>1$.  

\begin{remark}
In fact in what follows we only need that $|v'_R(r)| \le C$ for a constant $C$ independent of $R$.  
In this case we can actually write down the solution $v_R$ explicitly and prove this directly.  For example, when $n=2$, $v_R(r)$ is given by
$$v_R(r) = \int_0^r \frac{1}{x+R} \left( \int_0^x (y+R) h(y) dy \right) dx + c_1 \log (r+R) + c_2$$
for constants $c_1$ and $c_2$ given by
$$c_1 = \frac{- 1 - \int_0^1 \frac{1}{x+R} \left( \int_0^x (y+R) h(y) dy \right) dx}{\log ((1+R)/R)}, \quad c_2 = 1- c_1 \log R.$$
To see that $|v_R'(r)| \le C$ for a 
constant $C$ independent of $R$, note that as $R \rightarrow \infty$ the term $$ \left| \frac{c_1}{r+R} \right| = \textrm{O}\left( \frac{1/R}{\log (1+\frac{1}{R})} \right) $$ remains bounded.  
\end{remark}

\bigskip

We can now start the proof of Theorem \ref{theorem2}.

\begin{proof}[Proof of Theorem \ref{theorem2}]
Define our radially symmetric function $u_0$ on $\overline{\Omega}$ by
 $$u_{0}(r) = v_R (r-R), \quad \textrm{for } R \le r \le R+1,$$
 which satisfies the boundary conditions $u_{0}(R) = 1$, $u_{0}(R+1)=0$.  Moreover, from (\ref{dpv}) we have $$\Delta u_{0} (r) = h(r-R), \quad \textrm{for } R \le r \le R+1.$$ Now notice that by the definition of $h$ we have for $r=R$ or $r=R+1$,
$$\Delta^k u_0 (r)= 0, \quad \textrm{ for every} \ k \geq 1.$$ 
Let $u(r,t)$ be 
the rotationally symmetric solution to (\ref{heat}) on $\Omega$ with initial condition $u_0$.  
Then from \cite[Theorem 5.2]{LSU} (or \cite[Theorem 10.4.1]{Kr}) it follows that the function $u(r,t)$ extends to a smooth function on $\overline{\Omega} \times [0,\infty)$.

Next we show that for sufficiently large $R$, the function $u_0$ satisfies the hypotheses of Proposition \ref{prop}.

\pagebreak[3]
\begin{lemma}\label{p1} For sufficiently large $R$ we have 
\begin{enumerate}
\item [(i)] $u_{0}(r)$ is strictly decreasing in $r$.
\item [(ii)] $u_{0}''(R+1/2) \ge 1/10$. 
\item[(iii)] $\Delta u_{0} >0$ for $R<r<R+1$.
\end{enumerate}
\end{lemma}
\begin{proof}
For  (i), note that for all $r\in [R,R+1]$ we have $|u'_0(r)|\leq C$ for some C independent of  $R$ as observed  above.  Thus for all $r\in [R, R+1]$ we have 
\[ \begin{split}
|u_0''(r)|\leq {} &  | \Delta u_0 (r)|+  (n-1)| u'_0(r)/r| \\  = {} &  | h(r-R)|+  (n-1)| u'_0(r)/r|  \\ \le  {} & 1/10+ (n-1)C/R \leq 1/5
\end{split} \]  as long as $R$ is larger than $10C(n-1)$.  On
 the
  other hand, by the boundary conditions on $u_0$ we have by the Mean Value Theorem that $u_0'(r_0)=-1$ at some point $r_0$ in
 $[R, R+1]$.  Hence  $u_0'(r)\le -1+1/5$ for all $r\in [R, R+1]$, giving (i).

For (ii), we have 
\[ \begin{split}
u_{0}''(R+1/2)=  {} & \Delta u_0 (R+1/2) - (n-1)u_{0}'(R+1/2)/(R+1/2) \\ \geq {} & \Delta u_0 (R+1/2) \\ = {}&  h(1/2)=1/10
\end{split} \]
 where in the second inequality we have used part (i) and in the third inequality we have used the definition of $u_0$ and the property (b) of $h$.

Finally (iii) follows from the definition of $h$.
\end{proof}

We now fix $R$ as in the lemma above.  Define $c=u_0(R+1/2)$ and let $u(x, t)$ be the solution to \eqref{heat}.   We consider the space-time superlevel set $$\Omega_c:=\{(x, t) \in \ov{\Omega} \times [0,\infty): u(x, t) \ge c\} \cup (\Omega_1\times [0,\infty)).$$    
  As noted above, $u(x, t)$ is smooth on $\ov{\Omega} \times [0,\infty)$ while from Lemma \ref{p1} and Proposition \ref{prop}, $u_t(x, t)$ is smooth and strictly positive on $\Omega \times [0,\infty)$.  By the Implicit Function Theorem, we may write $\partial \Omega_c =\{(x, t) \in \ov{\Omega} \times [0,\infty): u(x, t)= c\}$ as the radial graph of the equation $t=f(r)$ where $u(r, f(r))=c$ and $f(r)$ is a smooth increasing function on $[ R+1/2, R+1/2+\ve)$ for some $\ve>0$.  Note that $f(R+1/2)=0$.  Differentiating the defining equation for $f$ we obtain
\begin{equation}
f'(r)=-\frac{u_r}{u_t}
\end{equation}
where the functions above are evaluated at $(r, f(r))$, and
\begin{equation}\begin{split}
f''(r)&=\frac{-1}{u_t}(u_{rr}+2 u_{rt}f' +  u_{tt}(f')^2 )\\
\end{split}\end{equation}
where again these are all evaluated at $(r, f(r))$.  Now evaluating the above at $(r, f(r))=(R+1/2, 0)$, and noting that $u_t(r, 0)=\Delta u_0(r)=h(r-R)=1/10$ in some neighborhood of $r=R+1/2$, hence $u_{tr}(R+1/2, 0)= u_{tt}(R+1/2, 0)=0$, we get 
\begin{equation}\begin{split}
f''(R+1/2)&=-10 u_0''(R+1/2) \le -1 <0 \\
\end{split}\end{equation}
by Lemma \ref{p1}.   This contradicts that $\partial \Omega_c$, which is defined by $t=f(r)$, is convex in $\mathbb{R}^{n+1}$.
\end{proof}

\section{Two-point functions} \label{section4}

In this section we discuss a different way to find a counterexample to space-time quasiconcavity using a two-point function as in \cite{W} (see also \cite{RR}).  We work in dimension $n\ge 1$ with the domain $\Omega =\{ 1 < r <2 \}$.  Let $u(x,t)$ be a smooth function on $\ov{\Omega} \times [0,\infty)$.
Consider the two-point function
$$\mathcal{H}((x,s), (y,t)) = (Du(y,t) - Du(x,s)) \cdot (y-x) + (u_t(y,t) - u_t(x,s))(t-s),$$
restricted to $(x,s), (y,t) \in \ov{\Omega} \times [0, \infty)$ with $u(x,s) = u(y,t)$.
If $u(x,t)$ is space-time quasiconcave then $\mathcal{H} \le 0$.  Indeed, note that for a smooth function $w$ in $\mathbb{R}^{n+1}$, if the superlevel set $\{ w >c \}$ is convex with a smooth boundary that contains two points $X$ and $Y$ then $(Dw(Y) - Dw(X)) \cdot (Y-X) \le 0$ since the vectors $Dw(X), Dw(Y)$ point in the inward normal direction. Applying this to $X=(x,s)$ and $Y=(y,t)$ in $\ov{\Omega} \times (0,\infty)$ with $u(x,s) = u(y,t)$, and using the continuity of $\mathcal{H}$, we have $\mathcal{H} \le 0$ on its domain of definition.

For our counterexample, we construct a radially symmetric admissible function $u_0$ such that the corresponding solution $u(x,t)$ of the  heat equation (\ref{heat}) is smooth on $\ov{\Omega} \times [0,\infty)$ and has 
$\mathcal{H}$  strictly positive somewhere.

Let $\ve>0$ be a small constant to be determined and let $g: [1,2] \rightarrow [0, \infty)$ be a smooth function with the following properties:
\begin{enumerate}
\item[(a)] $g(r) = 1/(2\ve)$ for $1 \le r \le 1+\ve$, and $g(r) =0$ for $2-\ve \le r\le 2$.
\item[(b)] $g$ is decreasing on $[1,2]$.
\item[(c)] $\displaystyle{\int_1^2 r^{1-n}g(r)dr =1}$. 
\end{enumerate}
We then define a smooth function $u_0$ on $[1,2]$ by
$$u_0(r) = - \int_1^r s^{1-n} g(s) ds +1, \quad r \in [1,2].$$
It is straightforward to check that $u_0$ is an admissible function. Moreover, $\Delta u_0$ vanishes identically in a neighborhood of $r=1$ and $r=2$.  Let $u(x,t)$ be the corresponding solution of (\ref{heat}), which by \cite[Theorem 5.2]{LSU} is smooth on $\ov{\Omega} \times [0,\infty)$.  From Proposition \ref{prop} we know that $u(x,t)$ converges smoothly as $t \rightarrow \infty$ to the harmonic function $u_{\infty}$ given by
$$u_{\infty} (r) = \left\{ \begin{array}{ll} 2-r, & \quad n=1 \\ 1- \frac{\log r}{\log 2}, & \quad n=2 \\ \frac{(2/r)^{n-2}-1}{2^{n-2}-1}, & \quad n>2.\end{array} \right.$$

We now choose our points $(x,s)$ and $(y,t)$.  We choose $x$ and $y$ to lie in the line $x_2=\cdots = x_n=0$.  Pick $x= (1+\ve/2, 0, \ldots, 0)$ and $s=0$.  By the definition of $u_0$ we have $u(x, 0) = u_0(1+\ve/2) \approx 3/4$.  
Let $\gamma$ solve $u_{\infty}(1+\gamma) = u(x,0)$, which satisfies $\gamma > c(n)$ for a constant $c(n) \in (0,1)$ depending only on $n$.  Writing $y(t) = (y_1(t), 0, \ldots, 0)$ solving $u(y(t),t) = u(x,0)$ we have $y_1(t) \rightarrow 1+\gamma$  and $u_r(y(t), t) \rightarrow (u_{\infty})_r (1+\gamma)$ as $t\rightarrow \infty$.  Then for $t$ sufficiently large and $\ve>0$ sufficiently small we have $y_1 >1+\ve/2$ and $|u_r(y(t),t)| \le C$ for a uniform $C$.  Writing $y$ for $y(t)$ we have
$$\mathcal{H}((x,s), (y,t)) = (u_r(y,t) - u_r(x,0))  (y_1- (1+\ve/2)) + (u_t(y,t) - u_t(x,0))(t-0).$$
From the definition of $u_0$ we have $u_r(x,0) \approx - 1/(2\ve)$ and $u_t(x, 0) = \Delta u(x,0) =0$.  On the other hand, from Proposition \ref{prop}, $u_t(y,t)>0$.  Hence for $\ve$ sufficiently small, $\mathcal{H}>0$.

\begin{remark}
As in \cite{W} one can show that a maximum principle holds for the quantity $\mathcal{H}$ using a parabolic version of a Lemma of Rosay-Rudin \cite{RR}.  This rules out $\mathcal{H}$ obtaining a positive interior maximum.  However, it does not rule out a positive maximum occuring at a point $((x,s), (y,t))$ with $s=0$ which would be needed to prove that quasiconcavity is preserved for the heat equation.
\end{remark}

\end{document}